\documentclass[11pt]{article}
\usepackage[T1]{fontenc}

\usepackage[utf8]{inputenc}

\usepackage{mathrsfs}
\usepackage{amsmath}
\usepackage{amsthm}
\usepackage{amssymb}
\usepackage{graphicx}
\usepackage{color}
\usepackage{caption,subcaption}
\usepackage{babel}
\usepackage{comment}
\usepackage{algorithm}
\usepackage{algpseudocode}
\usepackage{setspace}

\usepackage{grffile}
\usepackage{enumerate}

\usepackage{todonotes}

\usepackage{hyperref}
\hypersetup{colorlinks=true,
	linkcolor=black,
	anchorcolor=black,
	citecolor=black}

\numberwithin{equation}{section}
\numberwithin{figure}{section}
\numberwithin{table}{section}

\newtheorem{THEOREM}{Theorem}[section]

\newtheorem{corollary}[THEOREM]{Corollary}

\newtheorem{proposition}[THEOREM]{Proposition}

\renewcommand{\a}{\alpha}
\renewcommand{\b}{\beta}

\newcommand{\R}{\mathbb{R}}

\newcommand{\lam}{\lambda}

\newcommand{\bx}{\mathbf{x}}

\newcommand{\bau}{\overline{u}}

\newcommand{\be}{\mathbf{e}}

\newcommand{\gam}{\gamma}
\newcommand{\bv}{\mathbf{v}}
\renewcommand{\th}{\theta}

\newcommand{\bw}{\mathbf{w}}

\newcommand{\eps}{\varepsilon}

\newcommand{\bm}{\mathbf{m}}
\newcommand{\bam}{\overline{m}}

\newcommand{\mN}{m^N}
\newcommand{\bmN}{\bm^N}

\newcommand{\br}{\mathbf{r}}
\newcommand{\Dx}{\Delta x}
\newcommand{\Dt}{\Delta t}

\newcommand{\jph}{j+\frac{1}{2}}
\newcommand{\jmh}{j-\frac{1}{2}}
\newcommand{\nph}{n+\frac{1}{2}}
\newcommand{\rd}{\mathrm{d}}

\newcommand{\tF}{\widetilde{F}}

\newcommand{\hHe}{\widetilde{\mathsf{He}}}
\newcommand{\cH}{\mathcal{H}}
\newcommand{\feq}{\mu}
\newcommand{\balpha}{\mathbf{\alpha}}
\newcommand{\energy}{\mathcal{E}}

\newcommand{\df}{\delta f}
\newcommand{\drho}{\delta \rho}
\newcommand{\laplace}{\mathcal{L}}
\newcommand{\basig}{\overline{\sigma}}
\newcommand{\He}{\mathsf{He}}
\newcommand{\Aop}{\mathcal{A}}

\newcommand{\normx}[1]{ \| #1 \|_{L_{x}^2}}
\newcommand{\normwx}[1]{ \| #1 \|_{w_v^{-1} L_x^2}}
\newcommand{\normw}[1]{ \| #1 \|_{w_v^{-1}}}
\newcommand{\wi}{e^{-\frac{v^2}{2}}}
\newcommand{\ww}{e^{\frac{v^2}{2}}}

\DeclareMathOperator*{\diag}{diag}

\title{Controlling instability in the Vlasov-Poisson system through moment-based optimization}
\author{Jingcheng Lu, Li Wang and Jeff Calder}
\date{}

\begin{document}
	
	\maketitle
	
	\begin{abstract}
		Controlling instability in plasma is one of the central challenges in fusion energy research. Among the various sources of instability, kinetic effects play a significant role. In this work, we aim to suppress the instability induced by kinetic effects by designing an external electric field. However, rather than directly solving the full kinetic Vlasov-Poisson system, we focus on a reduced-order model, specifically the moment-based system, to capture the underlying dynamics. This approach is motivated by the desire to reduce the computational cost associated with repeatedly solving the high-dimensional kinetic equations during the optimization of the electric field. Additionally, moment-based data is more readily accessible in practice, making a moment-based control framework more adaptable to data-driven scenarios. We investigate the effectiveness of moment-based control both analytically and numerically, by comparing it to control based on the full kinetic model.
	\end{abstract}
	
	\tableofcontents
	
	\section{Introduction}
	Consider the Vlasov-Poisson (VP) system for the motion of charged particles (e.g. electrons) in plasma state: 
	\begin{equation}\label{eq:vlasov poisson}
		\left\{
		\begin{array}{l}
			\partial_t f(\bx,\bv,t) +\bv\cdot\nabla_\bx f(\bx,\bv,t)+E(\bx,t)\cdot\nabla_\bv f(\bx,\bv,t) = 0 \,, \\
			\\
			E(\bx,t) = -\nabla_\bx\phi(\bx,t)\,, \\
			\\
			\nabla_\bx\cdot E(\bx,t) = -\Delta_\bx\phi(\bx,t) = \rho(\bx,t)-\rho_{ion} \,,\\
			\\
			\rho(\bx,t) = \int f(\bx,\bv,t) d\bv\,.
		\end{array}
		\right.
	\end{equation}
	Here $f(\bx,\bv,t)$ is particle density at location $\bx$, time $t$, with velocity $\bv$. We assume that the spatial variable $\bx$ lies in a periodic box and the velocity $\bv \in \mathbb R^d$. This setup is consistent with the modeling of plasma in fusion devices such as in the work of Tokamaks  \cite{degrave2022magnetic}.  
	Here, $E(\bx,t)$ is the self-generated electric field, obtained as the negative gradient of a potential function $\phi(\bx)$, which is determined by the charge density $\rho(\bx,t)$ through the Poisson equation in the third line of \eqref{eq:vlasov poisson}. The variable $\rho_{ion}$ is the number density of ions, which neutralize the plasma.   
	
	The VP system has been a rich subject of mathematical analysis, with significant work devoted to well-posedness and quantitative properties of  solutions \cite{rein2007collisionless}. Global well-posedness of classical solutions for large initial data was established in concurrent foundational works by Lions and Perthame \cite{lions1991propagation} and Pfaffelmoser \cite{pfaffelmoser1992global}, with different techniques. Among the properties of the solutions, one intriguing aspect is the stability (or instability) of the equilibrium upon perturbation. In particular, it is direct to verify that any function depending solely on $\bv$, i.e., $\mu(\bv)$, constitutes an equilibrium of \eqref{eq:vlasov poisson}, provided the neutrality condition
	\[
	\rho_{ion} = \int_{\mathbb R^d} \mu (\bv) \rd \bv
	\]
	is satisfied.  
	
	When $\mu(\bv)$ is a Maxwellian distribution, i.e., $\mu(\bv) \propto e^{-{|\bv-u|^2/2T}}$ for some given bulk velocity $u$ and temperature $T$, it is well-known that this equilibrium is stable, in the sense that small perturbations in plasma beams lead to the damping of electrostatic waves, a phenomenon known as Landau damping \cite{Landau, PhysRevLett.13.184, mouhot2011landau}. In contrast, when $\mu(\bv)$ is a mixture of two Gaussians, the equilibrium becomes unstable; that is, an initial disturbance in the beams is amplified, leading to undesirable growth of the perturbation and the transfer of energy from the plasma beams to electrostatic waves. If the two Gaussians have comparable magnitudes and are centered opposite about the origin, the resulting instability is referred to as the two-stream instability \cite{chen2016}. Alternatively, if the Gaussians have disparate magnitudes, with the stronger one concentrated near the origin and the weaker one at a larger velocity, the configuration is often termed a bump-on-tail distribution, which is likewise unstable.
	We also note that \eqref{eq:vlasov poisson} can admit spatially inhomogeneous equilibrium, which may also be unstable \cite{guo1995instability}. However, in this work, we focus exclusively on spatially homogeneous equilibria of the form  $\mu(\bv)$.

	These instabilities can pose significant challenges in plasma applications. For instance, the two-stream instability, commonly observed in particle accelerators, can lead to unwanted beam scattering and loss of focus \cite{sydorenko2016effect}. Similarly, in fusion devices, the bump-on-tail instability can arise due to runaway electrons or radiofrequency heating, resulting in a degradation of plasma heating efficiency. It can also cause loss of confinement and the deposition of large amounts of energy onto the reactor walls. Therefore, it is desirable to control these instabilities through carefully designed external forces. %electric or magnetic fields.
	
	The controllability of VP system  was first investigated by Glass and Han-Kwan in~\cite{glass2003controllability,glass2012controllability}, where the control is introduced as an external force term. They showed that, for a given initial condition $f_0(\bx,\bv)$, there exists a control that steers the solution to a prescribed final profile 
	$f_1(\bx,\bv)$ at time $t=T$. Later, in a series of papers~\cite{knopf2018optimal,knopf2019confined,knopf2020optimal}, Knopf and collaborators studied a similar problem: steering the plasma shape at the final time $t=T$ by tuning an external magnetic field. Their main tool was the calculus of variations, through which they analyzed the properties of the control-to-solution map and established necessary and sufficient conditions for the existence of an optimal control. More recently, in \cite{bartsch2023controlling}, they also addressed the problem numerically using a Monte Carlo-based particle-in-cell method. Additionally, Albi et al.~\cite{albi2025instantaneous} recently investigated instantaneous control strategies within the same problem setting.
	
	Another approach aimed at suppressing instability was recently explored in \cite{einkemmer2024suppressing}, where an external electric field is applied to stabilize the plasma beam and maintain its initial shape. More specifically, replacing $E$ in \eqref{eq:vlasov poisson} with $E + H$, where $H$ is the external control field, then the goal is to find an optimal $H$ such that the deviation of $f$ from the equilibrium $\mu$ at time $t=T$ is as small as possible: 
	\begin{align} \label{obj}
		\min_{H\in L^2_{x,t}} ||f(\cdot,\cdot,T;H)-\feq(\cdot)||_{L^2_{x,v}}\,.
	\end{align}
	In \cite{einkemmer2024suppressing}, this question is addressed within the PDE-constrained optimization framework. The authors first parameterize $H$ via $H(\bx,t;\balpha) = \sum_k \alpha_k\psi_k(\bx)$, and then determine the coefficients through a hybrid of local and global optimizers. The resulting time independent control is shown to effectively suppress plasma perturbation within the time interval $[0, T]$, but its effectiveness may diminish beyond $T$.
	Subsequently, the same problem was revisited using a linear stability analysis approach \cite{einkemmer2024control}, which leads to the identification of an explicit, time-dependent control based on the precise form of the initial instability.
	
	In this paper, we re-examine the problem \eqref{obj}, but instead of using the full VP system \eqref{eq:vlasov poisson} as the constrained equation, we employ a reduced-order surrogate, namely the hydrodynamic or, more generally, the moment system. This approach is motivated by the concern that optimizing against the full kinetic system \eqref{eq:vlasov poisson} may be intractable due to the curse of dimensionality. Instead, assessing the validity of the control at the hydrodynamic level offers a more computationally efficient approach. Additionally, in real applications, direct data on the particle distribution is often difficult to obtain, whereas moment-based data is more accessible. Therefore, a moment-based control framework is more easily adaptable to a data-driven scenario.

	To facilitate a convenient comparison between kinetic and moment control, we focus on the one-dimensional setting, in which case the VP system \eqref{eq:vlasov poisson} reduces to
	\begin{equation}\label{eq:1d vlasov poisson}
		\left\{
		\begin{array}{l}
			\partial_t f +v\partial_x f+(E+H)\partial_v f = 0 \,, \\
			\\
			E = -\partial_x\phi, \quad \partial_x E = -\partial^2_x\phi = \rho-\rho_{ion} \,, \quad 
			\rho = \int f \rd v\,.
		\end{array}
		\right.
	\end{equation}
	The first three moments of $f$ yield macroscopic quantities defined as:  
	\begin{equation}\label{eq:hydro mom}
		\begin{split}
			\rho(x,t) &:= \int_\R f(x,v,t) \rd v, \\
			\rho u(x,t) &:= \int_\R vf(x,v,t)\rd v,\\
			\energy(x,t) &:= \int_
			R\frac{v^2}{2}f(x,v,t)\rd v = \frac{1}{2}p+\frac{1}{2}\rho u^2, ~~ \text{where}~ p(x,t) := \int_\R (v-u)^2 f(x,v,t) \rd v\,.
		\end{split}
	\end{equation}
	Integration of \eqref{eq:1d vlasov poisson} leads to the  Euler-Poisson equations:
	\begin{equation}\label{eq:Euler Poisson}
		\left\{
		\begin{array}{l}
			\partial_t \rho +\partial_x(\rho u) = 0 \,, \\
			\\
			\partial_t \rho u +\partial_x(\rho u^2+p) = \rho (E+H) \,,\\
			\\
			\partial_t \energy+\partial_x\big((\energy+p)u+q\big) = 0 \,,\\
			\\
			E = -\partial_x\phi, \quad \partial_xE = -\partial^2_x\phi = \rho-\rho_{ion}\,.
		\end{array}
		\right.
	\end{equation}
	% or, equivalently, 
	% \begin{equation}\label{eq:Euler Poisson2}
		% \left\{
		% \begin{array}{l}
			% \partial_t \rho +\partial_x(\rho u) = 0  \\
			% \\
			% \partial_t u + u\partial_x u = -\frac{1}{\rho}\partial_x p+E+H\\
			% \\
			% \partial_t p+u\partial_x p +2\partial_x q = -2 p\partial_x u\\
			% \\
			% E = -\partial_x\phi, \quad \partial_xE = -\partial^2_x\phi = \rho-\rho_{ion}
			% \end{array}
		% \right..
		% \end{equation}
	Here, the heat flux $\displaystyle q(x,t) = \int \frac{(v-u)^3}{2}f(x,v,t)\rd v$ is associated with the third moment of particle distribution $f$. It can be expressed by $\rho$, $u$, $\energy$ (or $p$) by assuming $f$ takes some equilibrium form, in which case \eqref{eq:Euler Poisson} is then closed.
	
	In general, stability in the fluid regime does not necessarily imply stability when kinetic effects are considered \cite{vogman2020two,chen2016}. This is because the VP system \eqref{eq:vlasov poisson} exhibits much more complex dynamics, where even spatially homogeneous equilibrium distributions can become unstable, as described above. Therefore, rather than applying control to \eqref{eq:Euler Poisson}, we propose considering a higher-order moment system and investigating its effectiveness as the number of moments increases. 
	
	More particularly, we consider the following control problem: 
	\begin{equation*}
		\begin{split}
			&\min_{H} \ \frac{1}{2}\sum^{\infty}_{n=0}||m_n(T;H)-\overline{m}_n||^2_2,\\
			\\
			s.t.\quad &  \partial_t m_n+\sqrt{n}\partial_x m_{n-1}+\sqrt{n+1}\partial_x m_{n+1} = \sqrt{n}(E+H)m_{n-1}, \quad n = 0,1,2\dots.
			%\\ & H(x,t;\alpha) = \sum_k \alpha_k \psi_k(x,t).
		\end{split}
	\end{equation*}
	where $m_n(t; H)$ and $\bar m_n$ are the moments of the solution $f(t,\cdot, T; H)$ and equilibrium $\mu (\cdot)$, respectively, and their specific definition will be made clear in the next section. $E$ is obtained through $m_0$, see \eqref{eq:optimization with moments}.

	In practice, we truncate the moment system at some finite $N$, and denote the resulting moments as $m_n^N$, to distinguish them from the full moments $m_n$. In the next section, we introduce the moment system based on Hermite polynomials and describe a closure rule for the truncated system. Although the truncated moment system provides an approximation to the original Vlasov–Poisson system, and this approximation improves as more moments are included, the impact of truncation on the control outcome is less clear. This issue is explored in section 3, where we compare the optimized electric fields obtained under constraints from the full Vlasov–Poisson system versus those from the truncated moment system. In section 4, we outline the adjoint state method for solving the moment based control problem, and discuss the optimization strategy. Numerical examples are provided in section 5, where we demonstrate the effectiveness of moment based control for two scenarios, two-stream instability and bump-on-tail instability.

	\section{Suppressing instability with moment systems}
	In this section, we first lay out our moment systems and then reformulate the control problem as a moment system-constrained optimization. Instead of directly taking the moments of \eqref{eq:vlasov poisson}, as is done previously to lead to \eqref{eq:Euler Poisson}, we derive the moment system using Hermite polynomials in $v$. This idea dates back to the 90's, where a small finite set of orthogonal polynomials are used to approximate $f$ in $v$ as a numerical method in the Galerkin framework \cite{schumer1998vlasov, tang1993hermite, bessemoulin2022stability}.
	
	More specifically, let $f$ be the solution of the Vlasov-Poisson system \eqref{eq:1d vlasov poisson} and formally expand it using basis function $\{\mathcal H_n (v)\}$:
	\[
	f(x,v,t) = \sum^\infty_{n=0}m_n(x,t)\mathcal{H}_n(v), \quad \mathcal{H}_n(v) =  \hHe_n(v)e^{-\frac{v^2}{2}},
	\]
	where
	\[
	\displaystyle\hHe_n(v) = \sqrt{\frac{1}{\sqrt{2\pi}n!} }\He_n(v) \qquad \text{with}~~~\displaystyle \He_n(v) = (-1)^ne^{\frac{v^2}{2}}\frac{\rd^n}{\rd v^n}e^{-\frac{v^2}{2}}
	\]
	is the normalized Hermite polynomial and satisfies \footnote{For unfamiliar readers, we point out that the choice of orthogonal polynomials needs to facilitate two aspects: 1. the appropriate domain of $v$, e.g., $v\in(-\infty,\infty)$, for Boltzmann equation and Vlasov-Poisson equation, and $v\in[-1,1]$ for one-dimensional radiative transfer equation; 2. the recursion relation of the orthogonal polynomials needs to fit the structure of the kinetic equation such that after substituting the expansion into the equation, we can get a moment system with simple structure. Taking these aspects into consideration, Hermite polynomials (orthogonal on $v\in(-\infty,\infty)$) and Legendre polynomials (orthogonal on $v\in[-1,1]$) have been the conventional choice in the derivation of moment systems of high order.}
	\begin{align*}
		\displaystyle\int_\R\hHe_m(v)\hHe_n(v) w(v)\rd v = \delta_{mn}\,, \quad w(v) = e^{\frac{v^2}{2}}\,.
	\end{align*}
	Here $m_n$ is termed $n-th$ moment and can be defined by 
	\[
	m_n(x,t) = \int_\R f(x,v,t) \hHe_n(v) \rd v = \sqrt{\frac{1}{\sqrt{2\pi}n!}}\int_\R f(x,v,t)\He_n(v)\rd v.
	\]

	We recall the first few Hermite polynomials:
	\[
	\begin{split}
		&\He_0(v) = 1, 
		~~\He_1(v) = v, 
		~~\He_2(v) = v^2-1, \\
		&\He_3(v) = v^3-3v,  
		~~\He_4(v) = v^4-6v^2+3,
		~~\cdots
	\end{split}
	\]
	Following the definition in \eqref{eq:hydro mom},  we have 
	\[
	\begin{split}
		&m_0 = \sqrt{\frac{1}{\sqrt{2\pi}}}\int_\R f \rd v = \sqrt{\frac{1}{\sqrt{2\pi}}}\rho, \\
		&m_1 = \sqrt{\frac{1}{\sqrt{2\pi}}}\int_\R v f \rd v = \sqrt{\frac{1}{\sqrt{2\pi}}}\rho u, \\
		&m_2 = \sqrt{\frac{1}{2\sqrt{2\pi}}}\int_\R (v^2-1) f \rd v = \sqrt{\frac{1}{2\sqrt{2\pi}}}(2\energy-\rho) = \sqrt{\frac{1}{2\sqrt{2\pi}}}(p+\rho u^2-\rho),\\
		&m_3 = \sqrt{\frac{1}{6\sqrt{2\pi}}}\int_\R (v^3-3v)f\rd v = \sqrt{\frac{1}{6\sqrt{2\pi}}}(2q+3pu+\rho u^3-3\rho u),\\
		&\vdots
	\end{split}
	\]

	The evolution of the moments can be derived from the recursion relations of Hermite polynomials. In particular, we have 
	\[
	\begin{split}
		&\He_{n+1}(v) = v\He_n(v)-\He_n'(v),\\
		&\He_{n+1}(v) = v\He_n(v)-n\He_{n-1}(v).
	\end{split}
	\]
	Hence the normalized polynomials $\hHe_n$ satisfy:
	\begin{equation}\label{eq:hermite recursion}
		\begin{split}
			&\sqrt{n+1} \hHe_{n+1}(v) = v\hHe_n(v)-\hHe_n'(v),\\
			&\sqrt{n+1} \hHe_{n+1}(v) = v\hHe_n(v)-\sqrt{n}\hHe_{n-1}(v).
		\end{split}
	\end{equation}
	Consequently, we have 
	\[
	\begin{split}
		\partial_t f &= \sum^{\infty}_{n=0}\partial_t m_n \cH_n(v),\\
		v\partial_x f &= \sum^{\infty}_{n=0}\partial_x m_n v\cH_n(v) \\
		& = \sum^{\infty}_{n = 0} \partial_x m_n (\sqrt{n+1}\cH_{n+1}(v)+\sqrt{n}\cH_{n-1}(v))\\
		& = \sum^{\infty}_{n = 1} \sqrt{n} \partial_x m_{n-1} \cH_{n}(v)+\sum^{\infty}_{n = 0}\sqrt{n+1}\partial_x m_{n+1}\cH_{n}(v),\\
		(E+H)\partial_v f &= (E+H)\sum^{\infty}_{n=0} m_n (\hHe_n'(v)-v\hHe_n(v))e^{-\frac{v^2}{2}}\\
		&= -(E+H)\sum^{\infty}_{n=0} \sqrt{n+1}m_n \hHe_{n+1}(v)e^{-\frac{v^2}{2}}\\
		& = -(E+H)\sum^{\infty}_{n=1} \sqrt{n} m_{n-1}\cH_{n}(v).
	\end{split}
	\]
	Substituting the above expression into \eqref{eq:1d vlasov poisson} yields the corresponding moment equations \footnote{Note that the definition of moments here is different from \eqref{eq:hydro mom}. While \eqref{eq:hydro mom} has a clearer physical interpretation for lower-order moments and is often used in Euler/Navier-Stokes equations, it leads to a highly complicated convection term and thus is inconvenient when deriving equations for higher-order moments. The definition through orthogonal polynomials, due to their recursion relations, will lead to a simpler structure of moment system and is a more common choice for deriving higher-order moment equations. In fact, the two sets of moments can be converted into each other.}:
	\begin{equation}\label{eq:VP moment eqn}
		\begin{split}
			&\partial_t m_0+\partial_x m_1 = 0,\\
			&\partial_t m_1+\partial_x m_0+\sqrt{2}\partial_x m_2 = (E+H)m_0,\\
			&\partial_t m_2+\sqrt{2}\partial_x m_1+\sqrt{3}\partial_x m_3 = \sqrt{2}(E+H)m_1,\\
			&\vdots\\
			& \partial_t m_N+\sqrt{N}\partial_x m_{N-1}+\sqrt{N+1}\partial_x m_{N+1} = \sqrt{N}(E+H)m_{N-1}\,.
		\end{split}
	\end{equation}
	Let $\bm_N = (m_0,m_1,\cdots,m_N)^\top$ consists of moments up to the $N-$th order, \eqref{eq:VP moment eqn} can be written into a more compressed form: 
	\begin{equation}\label{eq:VP moment sys}
		\frac{\partial \bm_N}{\partial t}+A_N\frac{\partial \bm_N}{\partial x} = -\sqrt{N+1}\frac{\partial m_{N+1}}{\partial x}\be_{N+1}+(E+H)D_N\bm_N,
	\end{equation}
	where $\be_{N+1} = (0,0,\cdots,0,1)^\top$ is the unit vector along the $(N+1)-$th dimension, and the coefficient matrices are
	\begin{equation}\label{eq:mom convec mat}
		A_N = \left[
		\begin{array}{ccccccc}
			0 & 1 & 0 & 0&\cdots&0&0\\
			1 & 0 & \sqrt{2} & 0&\cdots&0&0\\
			0& \sqrt{2} & 0 & \sqrt{3} & \cdots&0&0\\
			\vdots &\vdots &\ddots &\ddots&\ddots& &\vdots\\
			\vdots& \vdots & &\ddots&\ddots&\ddots&\vdots\\
			0& 0& 0 &\cdots &\sqrt{N-1} &0 &\sqrt{N}\\ 
			0& 0& 0 &\cdots &0 &\sqrt{N} &0
		\end{array}
		\right],
	\end{equation}
	and 
	\begin{equation}\label{eq:mom source mat}
		D_N = \left[
		\begin{array}{ccccccc}
			0 & 0 & 0 & 0&\cdots&0&0\\
			1 & 0 & 0 & 0&\cdots&0&0\\
			0& \sqrt{2} & 0 & 0 & \cdots&0&0\\
			0& 0& \sqrt{3}& 0 &\cdots &0&0\\
			\vdots&\vdots&\vdots&\ddots&\ddots&\vdots&\vdots\\
			\vdots&\vdots&\vdots& &\ddots&0&0\\
			0& 0& 0 & \cdots&\cdots &\sqrt{N} &0
		\end{array}
		\right]\,.
	\end{equation}
	Due to the symmetry of $A_N$, the hyperbolicity of the system is guaranteed, and all eigenvalues will be non-zero when $N$ is odd\footnote{When $N$ is even, $A_N$ will have eigenvalue 0. Indeed, if we denote $A_{k:N}$ to be the tri-diagonal matrix zero diagonal values, and off diagonal elements being $\sqrt{k}$ to $\sqrt{N}$. Then one can show that $\text{det}(A_{k:N}) = k \text{det}(A_{k+2:N})$. By induction, it boils down to either a $2\times 2$ matrix $\begin{bmatrix}
			0 & \sqrt{N} \\
			\sqrt{N} & 0
		\end{bmatrix}$ when $N$ is odd or 
		% $\begin{bmatrix}
			% 0 & \sqrt{N-1} & 0 \\
			% \sqrt{N-1} & 0 & \sqrt{N} \\
			% 0 & \sqrt{N} & 0
			% \end{bmatrix}$ 
		$\begin{bmatrix}
			0
		\end{bmatrix}$
		when $N$ is even. 
	}.

	If we also expand the equilibrium state $\displaystyle \feq(x,v) = \sum^\infty_{n=0} \bam_n(x)\mathcal{H}_n(v)$, we can then bound the $L^2$ perturbation of $f$ against $\feq$ with moments, that is, 
	\begin{equation}\label{eq:L^2 perturb bound}
		\begin{split}
			\int\int|f(x,v,t)-\feq(x,v)|^2\rd v\rd x &= \int\int |\sum^\infty_{n=0} (m_n(x,t)-\bam_n(x))\mathcal{H}_n(v)|^2\rd v\rd x\\
			& = \int\int |\sum^\infty_{n=0} (m_n(x,t)-\bam_n(x))\hHe_n(v)|^2e^{-v^2}\rd v\rd x\\
			& \leq \int\int |\sum^\infty_{n=0} (m_n(x,t)-\bam_n(x))\hHe_n(v)|^2e^{-\frac{v^2}{2}}\rd v\rd x\\
			& = \sum^\infty_{n=0}\int\int (m_n(x,t)-\bam_n(x))^2\hHe^2_n(v)e^{-\frac{v^2}{2}}\rd v\rd x\\
			& =  \sum^\infty_{n=0} ||m_n(\cdot,t)-\bam_n(\cdot)||^2_2.
		\end{split}
	\end{equation}
	%\lw{if we consider LHS to be $\int\int|f(x,v,t)-\feq(x,v)|^2 w(v)^{-1}\rd v\rd x $, then it equals the RHS}
	The fourth line is obtained due to the orthogonality of $\{\hHe_n(v)\}$ with respect to the weight function $\displaystyle e^{-\frac{v^2}{2}}$. This estimate indicates that if the perturbations of moments are sufficiently suppressed, then the perturbation of the underlying particle distribution will also be suppressed.

	Based on the above derivations, we summarize our control problem as follows:
	\begin{equation}\label{eq:optimization with moments}
		\begin{split}
			&\min_{\alpha} \ \frac{1}{2}\sum^{\infty}_{n=0}||m_n(T;H)-\overline{m}_n||^2_2,\\
			\\
			s.t.\quad &  \partial_t m_k+\sqrt{k}\partial_x m_{k-1}+\sqrt{k+1}\partial_x m_{k+1} = \sqrt{k}(E+H)m_{k-1}, \quad k = 0,1,2\dots\,;
			\\ & E = -\partial_x\phi, \quad \partial_x E = -\partial^2_x\phi = \rho-\rho_{ion} \,, \quad 
			\rho = (2\pi)^{1/4} m_0\,;
			\\ & H(x,t;\alpha) = \sum_k \alpha_k \psi_k(x,t).
		\end{split}
	\end{equation}
	Here $\psi_k(x,t)$ is a user specified basis that may depend on the prior knowledge of the available control in real experimental set up. 
	
	As a comparison, we also recall the original VP constrained control in \cite{einkemmer2024suppressing} as follows (see also \eqref{obj}): 
	\begin{equation}\label{eq:optimization with VP}
		\begin{split}
			&\min_{\alpha} \ \frac{1}{2}||f(T;H)-\feq||^2_{L^2_{x,v}},\\
			\\
			s.t.\quad & 
			\left\{\begin{array}{l}
				\partial_t f+v\partial_x f+(E+H)\partial_v f = 0 \\
				\\
				E = -\partial_x \phi, \quad  -\partial^2_x\phi  = \rho-\rho_{ion}, \quad \rho(x,t) = \int f(x,v,t)\rd v,\\
				\\
				H(x,t;\balpha) = \sum_k \alpha_k \psi_k(x,t).
			\end{array}
			\right.
		\end{split}
	\end{equation}

	As written, directly solving the problem \eqref{eq:optimization with moments} faces two main challenges:
	
	\noindent\textbf{1.} The objective function is expressed as an {\it infinite} sum of moment perturbations, which is not accessible in practical computations.
	
	\noindent\textbf{2.} The original Vlasov-Poisson dynamics corresponds to an \emph{infinite} set of moment equations. However, in practical applications, the moment system is truncated at a finite order, leading to the closure problem for the highest moment.

	To address the first issue, a straightforward approach is to truncate the infinite moment perturbations at order $N$. Naturally, including more moments in the cost function ensures that the perturbation of the particle distribution is more effectively controlled by the moment perturbations. The moment closure problem, however, is more challenging to resolve, as an analytical closure is often unavailable in complex applications, and data-driven approaches are sometimes required \cite{huang2022machine1}. In the context of perturbation suppression, however, a reasonable choice for closure is to use the target equilibrium, that is, 
	\[
	\partial_x m_{N+1}\approx \partial_x \overline{m}_{N+1},
	\]
	with $\overline{m}_{N+1}(x) = \int \feq(x,v) \hHe_{N+1}(v) \rd v$. This makes sense because, when the initial perturbation is not too large and with the assistance of an appropriately introduced external field, we expect the solution to remain close to the equilibrium,  making $\partial_x \overline{m}_{N+1}$ a good approximation. 
	
	Consequently, we further approximate  \eqref{eq:optimization with moments} to arrive at our final moment-based control problem: 
	\begin{subequations}\label{eq:optimization with truncated moments}
		\begin{equation}\label{eq:optim target}
			\min_{\alpha} \ \frac{1}{2}\sum^N_{n=0} ||\mN_n(T;\alpha)-\overline{m}_n||^2_2,
		\end{equation}
		subject to the truncated moment system:
		\begin{equation}\label{eq:optim constraint}
			\left\{\begin{array}{l}
				\partial_t \bmN_N+A_N\partial_x \bmN_N+\sqrt{N+1}\partial_x \overline{m}_{N+1}\be_{N+1} = (E_N+H)D_N\bmN_N,\\
				\\
				E_N = -\partial_x \phi_N, \quad  -\partial^2_x \phi_N = \rho_N-\rho_{ion}, \quad \rho_N(x,t) = (2\pi)^{\frac{1}{4}}\mN_0(x,t),\\
				\\
				H(x,t;\alpha) = \sum_k \alpha_k \psi_k(x,t).
			\end{array}
			\right.
		\end{equation}
	\end{subequations}
	Here $\bmN_N = (\mN_0,\mN_1,\cdots,\mN_N)$ is the vector of approximate moments. In particular, when a spatially homogeneous equilibrium is considered, we have $\partial_x \overline{m}_{N+1} = 0$.

	\section{Multiscale analysis of the truncated system}
	
	This section is devoted to investigating the connections between the kinetic-based control \eqref{eq:optimization with VP} and moment-based control \eqref{eq:optimization with truncated moments}, particularly the relationship between their respective optimizers $H_{VP}$ and $H_{mom}$.
	
	In general, the difference between $H_{VP}$ and $H_{mom}$ originates from two main sources: 
	
	\noindent\textbf{1. Target error}: the discrepancy between the distribution perturbation $\displaystyle \frac{1}{2}||f-\feq||^2_{L_{x,v}^2}$ 
	%\lw{maybe change the norm to $\normwx{\cdot}$ here}, 
	and the moment perturbations up to order $\displaystyle\frac{1}{2}\sum^N_{n=0}||\mN_n-\bam_n||^2_{L_x^2}$; 
	
	\noindent\textbf{2. Model error}: the deviation between the exact Vlasov–Poisson system and the truncated moment system \eqref{eq:optim constraint}.

	\subsection{Target error}
	The target error arises from the truncation of the infinite series $\sum_{n=0}^\infty \normx{m_n- \bar m_n}^2 = \normwx{f-\mu}^2$ to a finite sum. This error generally depends on the regularity of the function.   Since the truncation occurs only in the $v$-variable, we first present the following proposition for functions depending solely on $v$. The extension to functions depending on both $x$ and $v$ is straightforward.
	
	\begin{proposition} \label{prop1}
		Consider a function $g(v)$ such that $\normw{g} < +\infty$. Expand $g$ as 
		\begin{align}
			g = \sum_{l=0}^\infty \hat g_l \hHe_l(v)\wi, \quad \text{where}~~ \hat g_l = \int_\R g(v) \hHe_l(v) \rd v\,.
		\end{align}
		Then define an operator $\Aop g := \partial_v g + v g$, we have:
		\begin{align}
			\sum_{l=N+1}^\infty |\hat g_l|^2 \lesssim N^{-k} \normw{\Aop^k g}^2\,.
		\end{align}
	\end{proposition}
	
	This proposition can be viewed as a variation of \cite[Theorem 7.13]{shen2011spectral} or \cite[Lemma 2.5]{guo1999error}, with the main difference being how we define the moments (or spectral coefficients). We also note that the bound we derive imposes a rather strong condition on $g$. Specifically, 
	$g$ must not only be sufficiently regular, as required by the repeated application of the operator $\Aop$, but it also needs to exhibit sufficient decay at the tail. This is because the 
	$k$-th weighted moments of $g$, weighted by the rapidly growing function $\ww$ must be bounded. One example of such a function 
	$g$ that satisfies this condition is a Gaussian distribution.
	
	\begin{proof}
		Note first that $\hHe$ satisfies the following relation
		$
		\frac{\rd }{\rd v} ( \wi \frac{\rd}{\rd v} \hHe_n) = -\lambda_n \wi \hHe_n$ with $ \lambda_n = n $, we have
		\begin{align*}
			\int_\R g(v) \hHe_n(v) \rd v &= -\frac{1}{\lambda_n} \int_\R g(v) \ww  \frac{\rd }{\rd v} ( \wi \frac{\rd}{\rd v} \hHe_n) \rd v
			\\ & = \frac{1}{\lambda_n} \int_\R \frac{\rd }{\rd v} (g(v) \ww )  ( \wi \frac{\rd}{\rd v} \hHe_n) \rd v
			\\ & = \frac{1}{\lambda_n} \int_\R  \Aop g \frac{\sqrt{c_{n-1}}}{\sqrt{c_n} } \lambda_n \hHe_{n-1} \rd v\,, \qquad c_n := n! \sqrt{2\pi}
		\end{align*}
		where the last equality uses the fact $\sqrt{c_n} \hHe'(v) = \lambda_n \sqrt{c_{n-1}} \hHe_{n-1}(v)$. Repeat the above process we get 
		\begin{align} \label{0614}
			g_n = \int_\R g(v) \hHe_n(v) \rd v &= \frac{\sqrt{c_{n-k}}}{\sqrt{c_n}} \int_\R \Aop^k g \hHe_{n-k} \rd v \,.
		\end{align}
		Consequently, 
		\begin{align*}
			\sum_{l=N+1}^\infty |\hat g_l|^2  & = \sum_{l=N+1}^\infty \frac{c_{l-k}}{c_l} \left( \int_\R \Aop^k g \hHe_{l-k} \rd v \right)^2 
			\\ & \leq \max_{l \geq N+1, l \geq k} \frac{c_{l-k}}{c_l} \sum_{l=N+1}^\infty \left( \int_\R \Aop^k g \hHe_{l-k} \rd v \right)^2 \lesssim N^{-k} \normw{\Aop^k g}^2\,.
		\end{align*}
		The last inequality is obtained by noticing that $\displaystyle ||g||^2_{\omega_v^{-1}} = \sum^{\infty}_{n=0}|\hat g_n|^2 = \sum^{\infty}_{n=0}\left(\int_\R g(v)\hHe_n(v)dv\right)^2$.
	\end{proof}
	
	As an immediate consequence of Proposition~\ref{prop1}, we have
	\begin{corollary} \label{col}
		Let $m_n(x,t) = \int_\R f(x,v,t) \hHe_n(v) \rd v$ and $\bar m_n = \int_\R \mu(v) \hHe_n(v) \rd v$, then 
		\begin{align*}
			\sum_{n={N+1}}^\infty \normx{m_n- \bam_n}^2 \lesssim N^{-k} \normwx{\Aop^k(f - \mu)}^2\,.
		\end{align*}
	\end{corollary}

	\subsection{Model error}
	To quantify the model error, we propose the following auxiliary function $f_N$, which maps the moment information to an approximate distribution function:
	\begin{equation}\label{eq:fN near-equi}
		f_N(x,v,t) = \feq(v)+\sum^{N}_{n=0}(\mN_n(x,t)-\bam_n)\cH_n(v).
	\end{equation}
	Clearly, following the lines of \eqref{eq:L^2 perturb bound}, we have
	\begin{equation}\label{eq:f_N L^2 bound}
		\frac{1}{2}||f_N-\mu||^2_{L^2_{x,v}} = \frac{1}{2} ||\sum^{N}_{n=0}(\mN_n-\bam_n)\cH_n||^2_{L^2_{x,v}}\leq \frac{1}{2}\sum^{N}_{n=0}||\mN_n-\bam_n||^2_2\,.
	\end{equation}
	Notably, the upper bound on the right-hand side corresponds exactly to the objective function \eqref{eq:optim target} in the moment-based optimization. This implies that the optimized external field, $H_{mom}$,  constrained by the truncated moment system, effectively suppresses the perturbation of the approximate distribution 
	$f_N$ from the equilibrium $\mu$. 
	
	To characterize the evolution of $f_N$ under ansatz \eqref{eq:fN near-equi}, we use the recursion \eqref{eq:hermite recursion}  and derive
	\begin{equation*}
		\begin{split}
			\partial_t f_N+v\partial_xf_N+(E_N+H)\partial_v f_N = &
			\sum^{N}_{n=0}[\partial_t\mN_n+\sqrt{n}\partial_x\mN_{n-1}+\sqrt{n+1}\partial_x\mN_{n+1}-(E_N+H)\sqrt{n}\mN_{n-1}]\cH_n\\
			&+\sqrt{N+1}\partial_x \mN_N\cH_{N+1}-(E_N+H)\sqrt{N+1}\mN_N\cH_{N+1}\\
			& +(E_N+H)\partial_v\big(\feq-\sum^{N}_{n=0}\bam_n\cH_{n}\big)\\
			=& \sqrt{N+1}\partial_x \mN_N\cH_{N+1}-(E_N+H)\sqrt{N+1}\mN_N\cH_{N+1}\\
			& +(E_N+H)\partial_v\big(\feq-\sum^{N}_{n=0}\bam_n\cH_{n}\big)
		\end{split}
	\end{equation*}
	Using the relation $\sqrt{n+1}\cH_{n+1}(v) = -\cH_{n}'(v)$, the above equation can be simplified to:
	\begin{equation}\label{eq:truncated kinetic eqn}
		\left\{\begin{array}{l}
			\partial_t f_N+v\partial_x f_N+(E_N+H)\partial_v(\displaystyle \sum^{N-1}_{n=0}\mN_n\cH_{n}) = \sqrt{N+1}\partial_x\mN_N\cH_{N+1},\\
			\\
			E_N = -\partial_x\phi_N, \quad -\partial^2_x\phi_N = \rho_N-\rho_{ion}\\
			\\
			\rho_N = \int f_N \rd v.
		\end{array}
		\right.
	\end{equation}
	In this way, the dynamics described by \eqref{eq:truncated kinetic eqn} provide an equivalent kinetic representation of the moment system \eqref{eq:optim constraint}, enabling a direct comparison with the exact Vlasov–Poisson system. 
	
	More specifically, we compare the two control problems:
	\begin{itemize}
		\item[1)] Our moment control: $\min_{H_{mom}} \|f_N - \mu\|^2$ subject to \eqref{eq:truncated kinetic eqn};
		\item[2)] Original control: $\min_{H_{VP}} \|f - \mu\|^2$ subject to \eqref{eq:1d vlasov poisson}.
	\end{itemize}
	For this purpose, we adopt a linearized approach, following the methodology used in the derivation of the Penrose condition~\cite{penrose1960electrostatic}, the proposal of Landau damping by Landau~\cite{Landau}, and, more recently, the work of Einkemmer et al.~\cite{einkemmer2024control} in designing controlled electric fields.

	Let $f$ and $f_N$ be the particle distributions in the true physics \eqref{eq:1d vlasov poisson} and the moment based system \eqref{eq:truncated kinetic eqn}, respectively. We assume $f$ and $f_N$ are small perturbations around the equilibrium $\mu$ as: 
	\[
	\df := f-\feq\ll \feq, \quad \df_N := f_N-\feq\ll \feq, \quad \mN_n-\bam_n\ll\bam_n.
	\]
	Then \eqref{eq:1d vlasov poisson} and \eqref{eq:truncated kinetic eqn} can be linearized into 
	\begin{equation}\label{eq:linearized vp}
		\partial_t \df+v\partial_x \df+(E+H)\partial_v \feq =0,
	\end{equation}
	and 
	\begin{equation}\label{eq:linearized truncated kinetic eqn}
		\partial_t \df_N+v\partial_x \df_N+(E_N+H)\partial_v \feq_{N-1} = \sqrt{N+1}\partial_x \mN_N\cH_{N+1}, \quad \feq_{N-1}(v) = \sum^{N-1}_{n=0}\bam_n\cH_n(v).
	\end{equation}
	Likewise, we define the perturbation densities,
	\[
	\drho(x,t) = \int \df(x,v,t)\rd v, \quad \drho_N(x,t) = \int \df_N(x,v,t)\rd v,
	\]
	and, for convenience of discussion, assume that 
	\[
	\df(x,v,0) = \df_N(x,v,0) = \df_0(x,v).
	\]
	
	By performing Fourier and Laplace transforms on equation \eqref{eq:linearized vp}, it was shown in \cite[Appendix B]{einkemmer2024control} that the density perturbation can be described with
	\begin{equation}\label{eq:vp linearized perturb}
		\laplace[\widehat{\drho}(\xi,\cdot)](s) = \laplace[\widehat{\drho_F}(\xi,\cdot)](s)-\frac{i\xi\laplace[\widehat{H}(\xi,\cdot)](s)+\laplace[\widehat{\drho_F}(\xi,\cdot)](s)}{1+\laplace[\widehat{U}(\xi,\cdot)](s)}\laplace[\widehat{U}(\xi,\cdot)](s),
	\end{equation}
	where $\widehat{\cdot}$ represents Fourier transform in space, $\laplace[\cdot]$ represents Laplace transform in time. The function $U(x,t)$ satisfies
	\begin{equation} \label{hatU}
		\widehat{U}(\xi,t) = t\widehat{\mu}(t\xi),
	\end{equation}
	and 
	\begin{align} \label{drhoF}
		\drho_F(x,t) = \int \df_F(x,v,t)\rd v
	\end{align}
	is the free stream perturbation density where $\df_F(x,v,t) = \df_0(x-vt,v)$ satisfies the free stream transport $\partial_t \df_F+v\partial_x\df_F = 0$.
	Applying similar derivations to \eqref{eq:linearized truncated kinetic eqn}, we obtain
	\begin{equation}\label{eq:truncated kinetic linearized perturb}
		\begin{split}
			\laplace[\widehat{\drho_N}(\xi,\cdot)](s) = \laplace[\widehat{\drho_F}(\xi,\cdot)](s)&-\frac{i\xi\laplace[\widehat{H}(\xi,\cdot)](s)+\laplace[\widehat{\drho_F}(\xi,\cdot)](s)}{1+\laplace[\widehat{U_{N-1}}(\xi,\cdot)](s)}\laplace[\widehat{U_{N-1}}(\xi,\cdot)](s)\\
			&+\frac{i\xi\sqrt{N+1}\laplace[\widehat{\mN_N}(\xi,\cdot)](s)}{1+\laplace[\widehat{U_{N-1}}(\xi,\cdot)](s)}\laplace[\widehat{\cH}_{N+1}(\xi\cdot)](s),
		\end{split}
	\end{equation}
	where 
	\[
	\widehat{U_{N-1}}(\xi,t) = t\widehat{\mu_{N-1}}(t\xi) = t\sum^{N-1}_{n = 0} \bam_n\widehat{\cH}_{n}(t\xi), \quad  \widehat{\cH}_{N+1}(t\xi) = \int \cH_{N+1}(v)e^{-i(\xi t)v}dv.
	\]

	Then, if we consider choosing the external field $H$ to cancel the self-generated electric field, as suggested in \cite[Section 3.1]{einkemmer2024control}, we obtain the following result concerning the difference between the two control problems.
	\begin{proposition}
		Let $H_{mom}$ and $H_{VP}$ denote the control fields for systems \eqref{eq:truncated kinetic eqn} and \eqref{eq:1d vlasov poisson} respectively. Suppose these fields are selected to cancel the self-generated electric field, that is,
		\begin{align}
			i\xi \laplace[\widehat{H}_{Vp}] &= - \laplace[\widehat{\delta \rho_F}] \label{0616}
			\\ i\xi \laplace[\widehat{H}_{mom}] &= -\laplace[\widehat{\delta \rho_F}] + \frac{i\xi \sqrt{N+1}}{\laplace [\widehat{U}_{N-1}]} \laplace[\widehat{m_N^N}] \laplace[\widehat{\mathcal{H}}_{N+1}] \label{0617}
		\end{align}
		where $\delta \rho_F$ is defined in \eqref{drhoF}, and $\laplace[\widehat{H}]$ is a short notation for $\laplace[\widehat{H}(\xi,\cdot)](s)$. Then, if the equilibrium $\widehat U$ (corresponding to $\mu$ by \eqref{hatU}) has the property that $|\laplace[\widehat U_{N-1}]|$ is bounded away from zero independent of $N$, the difference between the two control fields satisfies
		\begin{align} \label{0618}
			|\laplace[\widehat{H}_{mom}] - \laplace[\widehat{H}_{VP}]| =  \left|\frac{\laplace[\widehat{\mathcal{H}}_{N+1}]  }{\laplace [\widehat{U}_{N-1}]} \sqrt{N+1} \laplace[ \widehat{m_N^N}] \right| \lesssim \frac{\sqrt{N+1}}{N^{k/2}} \normwx{\Aop^k f_N}\,, 
		\end{align}
		where $\Aop$ is again defined by $\Aop g := \partial_v g + vg$. 
		Additionally, if we use the $\widehat H_{mom}$ obtained from moment control system to control the original Vlasov-Poisson system, we obtain the following perturbation:
		\[
		\laplace[\widehat{\drho}] = \laplace[\widehat{\drho_F}]-
		\frac{\laplace[\widehat U]}{1+ \laplace[\widehat U]} \frac{i\xi \laplace[\widehat{\mathcal H}_N] }{\laplace[\widehat U_{N-1}]}\sqrt{N+1} \laplace[\widehat{m_N^N}].
		\]
	\end{proposition}
	\begin{proof}
		The derivation of \eqref{0616} and \eqref{0617} is straightforward from \eqref{eq:vp linearized perturb} and \eqref{eq:truncated kinetic linearized perturb}, respectively. The estimate \eqref{0618} follows from the fact that $|\mathcal{H}_{N}| \lesssim e^{-\frac{v^2}{4}}$ and the application of Corollary~\ref{col}. 
	\end{proof}

	% It is observed from equations \eqref{eq:vp linearized perturb} and \eqref{eq:truncated kinetic linearized perturb} that the model error of the truncated system \eqref{eq:optim constraint} (or equivalently \eqref{eq:truncated kinetic eqn}), as measured by the difference between $\drho$ and $\drho_N$, arises from the truncation error of $\mu_{N-1}$ against $\mu$ and the additional moment gradient term on the right of \eqref{eq:truncated kinetic eqn}. Especially, if the external field is chosen to cancel the free stream perturbation,
	% \[
	% i\xi \laplace[\widehat{H}] = -\laplace[\widehat{\drho_F}]\leadsto \partial_x H = -\drho_F,
	% \]
	% taking difference of \eqref{eq:vp linearized perturb} and \eqref{eq:truncated kinetic linearized perturb} leads to 
	% \[
	% \laplace[\widehat{\drho_N}(\xi,\cdot)](s) = \laplace[\widehat{\drho}(\xi,\cdot)](s)+\frac{i\xi\sqrt{N+1}\laplace[\widehat{\mN_N}(\xi,\cdot)](s)}{1+\laplace[\widehat{U_{N-1}}(\xi,\cdot)](s)}\laplace[\widehat{\cH}_{N+1}(\xi\cdot)](s).
	% \]
	% In this case, the model error is dictated by the $N-$th moment, $\mN_N$, of $f_N$ and the equilibrium moments, $\bam_n$, up to order $n = N-1$ (which are encoded in the function $U_{N-1}$). The moment information of order $n\geq N+1$ turns out to be irrelevant.
	
	\section{Implementation of moment-based optimization}
	This section is devoted to solving the moment-based control problem \eqref{eq:optimization with truncated moments}, using the classical adjoint state method. The implementation details are outlined as follows. 
	
	\subsection{Derivation of the adjoint state method}
	We consider the constraint moment system,
	\begin{equation}\label{eq:VP moment sys truncated}
		\begin{split}
			\frac{\partial \bm}{\partial t}+A\frac{\partial \bm}{\partial x}+\sqrt{N+1}\frac{\partial \overline{m}_{N+1}}{\partial x} \be_{N+1} = (E_N+H)D \bm, 
		\end{split}
	\end{equation}
	where $\bm = (\mN_0,\mN_1,\cdots,\mN_N)^\top$, $H(x,t;\alpha) = \sum_k \alpha_k\psi_k(x,t)$, and the constant coefficient matrices $A$, $D$ are given in \eqref{eq:mom convec mat} and \eqref{eq:mom source mat}. Here we omit the  $N$ in the superscripts and subscripts of $\bmN_N$, $A_N$, $D_N$ for notation simplicity.
	
	Our goal is to optimize the external field $H(x,t;\alpha)$ such that the $L^2-$moment perturbation can be sufficiently suppressed: 
	\[
	L(\{\bm(T;\alpha)\}) = \frac{1}{2}||\bm(T;\balpha)-\overline{\bm}||^2_2 = \sum^N_{n=0}\frac{1}{2}\int |m_n(x,T;\balpha)-\overline{m}_n|^2\rd x,
	\]
	where $\overline{m}_n = \int \feq(v)\hHe_n(v)\rd v$ is the $n-$th moment of the desired equilibrium $\feq$. We intend to optimize the parameters $\balpha$ using the gradient descent iterations. To this end, we introduce the Lagrangian multiplier $\lambda(x,t)$ and set
	\[
	\begin{split}
		\mathscr{L}(\{\bm(T;\balpha)\}) &= L(\{\bm(T;\balpha)\})+\int^T_0\int \lambda(x,t)^\top (\partial_t \bm+A\partial_x\bm+\sqrt{N+1}\partial_x\overline{m}_{N+1} \be_{N+1}-(E_N+H)D\bm)\rd x\rd t.
	\end{split}
	\]
	Obviously, when $\bm(x,t)$ satisfies the equation \eqref{eq:VP moment sys truncated} we have $\mathscr{L}(\{\bm(T)\}) = L(\{\bm(T)\})$ for any arbitrary choice of $\lambda$. 
	
	Taking the gradient of $\mathscr{L}$ with respect to $\a_k$, we have
	\begin{equation}\label{eq:cont adjoint ibp}
		\begin{split}
			\frac{\partial \mathscr{L}}{\partial \a_k} &= \int (\bm(x,T)-\overline{\bm})^\top\partial_{\a_k}\bm(x,T) \rd x + \int^T_0\int \lambda(x,t)^\top (\partial_t \partial_{\a_k}\bm+A\partial_x\partial_{\a_k}\bm-(E_N+H)D\partial_{\a_k}\bm)\rd x\rd t\\
			& \quad -\int^T_0\int \lambda(x,t)^\top \big(\partial_{\a_k}E_N(x,t)+\psi_k(x,t)\big) D\bm(x,t) \rd x\rd t\\
			& = \int (\bm(x,T)-\overline{\bm}+\lambda(x,T))^\top\partial_{\a_k}\bm(x,T)\rd x-\int^T_{0}\int\big\{\partial_t\lambda^\top+\partial_x\lambda^\top A+\lambda^\top(E_N+H)D\big\}\partial_{\a_k}\bm \rd x\rd t\\
			&\quad -\int^T_0\int \lambda(x,t)^\top \big(\partial_{\a_k}E_N(x,t)+\psi_k(x,t)\big) D\bm(x,t) \rd x\rd t\,,
		\end{split}
	\end{equation}
	where the second equality uses integration by parts. 
	To avoid the expensive evaluation of $\partial_{\a_k} \bm$, we let $\lambda$ solving the following space-time continuous adjoint equations
	\begin{equation}\label{eq:cont adjoint pde}
		\left\{\begin{array}{l}
			\displaystyle
			\partial_t\lambda(x,t) + A^\top \partial_x \lambda(x,t)= -\big(E_N(x,t)+H(x,t)\big)D^\top\lambda(x,t),  \\
			\\
			\lambda(x,T) = -(\bm(x,T)-\overline{\bm}). 
		\end{array}
		\right.
	\end{equation}
	Then the gradient of the loss function reduces to
	\begin{equation}\label{eq:cont adjoint gradient}
		\frac{\partial L}{\partial \a_k} = \frac{\partial \mathscr{L}}{\partial \a_k} = -\int^T_0\int \lambda(x,t)^\top \big(\partial_{\a_k}E_N(x,t)+\psi_k(x,t)\big) D\bm(x,t) \rd x\rd t.
	\end{equation}
	
	Note that in the gradient expression \eqref{eq:cont adjoint gradient}, the term $\partial_{\a_k} E_N$  is the most computationally involved. However, to better understand this term, consider its underlying structure. Let $G(x,x')$ be the Green's function to the 1D Poisson equation with appropriate boundary conditions. Then we have 
	\[
	\partial_{\a_k}E_N(x,t) = -\int \partial_x G(x,x') \partial_{\a_k}\delta \rho_N(x',t) \rd x'\,,
	\]
	where $\delta \rho_N = \rho_N-\overline{\rho}$ is the density perturbation. Hence, under the assumption of small perturbations, we have $|\partial_{\a_k}E_N(x,t)|\ll 1$ and can be safely omitted. Consequently, we will only use the following approximate gradient in the optimization: 
	\begin{equation}\label{eq:approx grad}
		\frac{\partial L}{\partial \a_k} \approx -\int^T_0\int \lambda(x,t)^\top \psi_k(x,t) D\bm(x,t) \rd x\rd t.
	\end{equation}
	Such simplification turns out to generate satisfactory results for moderate perturbations, as will be shown in section \ref{sec:numerical experiment}. When more accurate gradient calculation is needed, modern scientific computing packages such as   \texttt{torchdiffeq} can be used for numerical auto-differentiation.\footnote{For more reference, see 
		https://github.com/rtqichen/torchdiffeq}
	
	To discretize the moment system \eqref{eq:VP moment sys truncated} as well as the the adjoint equations \eqref{eq:cont adjoint pde}, we will use the semi-Lagrangian method combined with Strang splitting. 
	Since the matrix $A$ is symmetric, we can diagonalize it in the form
	\[
	A = R\Lambda R^{-1}, \quad R = [\br_0,\br_1,\cdots,\br_{N}],\quad \Lambda = \diag(\lambda_0,\lambda_1,\cdots,\lambda_{N}),
	\]
	where $\lambda_0,\dots,\lambda_{N}$ are the eigenvalues of $A$, $\br_0,\cdots,\br_{N}$ are the corresponding eigenvectors. They are obtained numerically with Matlab's \textbf{eig} solver. We define the characteristic variables 
	\[
	\bw(x,t) = R^{-1} \bm(x,t).
	\]
	Multiplying \eqref{eq:VP moment sys truncated} by $R^{-1}$ on the left yields
	\[
	\frac{\partial }{\partial t}\bw(x,t)+\Lambda\frac{\partial }{\partial x} \bw(x,t)= \tF(x,t), \quad \tF(x,t) = \big(E(x,t)+H(x,t)\big)R^{-1}D\bm(x,t),
	\]
	or equivalently, 
	\begin{align} \label{0610}
		\partial_t w_k(x,t)+\lambda_k\partial_x w_k(x,t) = \tF_k(x,t), \quad k = 0,1,\cdots,N,
	\end{align}
	where $w_k$ and $\tF_k$ are the $k-$th component of $\bw$ and $\tF$. Applying semi-Lagrangian to \eqref{0610} (or equivalently \eqref{eq:VP moment sys truncated}) with grid points $\{x_j\}$ and time step $\Dt$, we have:
	\begin{itemize}
		\item Compute 
		\[
		w^{\nph}_k(x_j) = w_k(x_j-\tfrac{\Dt}{2}\lambda_k,t^n), \ k = 0,1,\cdots N ~~\Longrightarrow ~~ \bm^{\nph}(x_j) = R\bw^{\nph}(x_j).
		\]
		\item Compute
		\[
		\bm^{n+1,*}(x_j) = \bm^{\nph}(x_j)+\Dt F^{\nph}(x_j),
		\]
		where $\displaystyle F^{\nph}(x_j) = \big(E_N^{{\nph}}(x_j)+H(x_j,t^{\nph})\big)D\bm^{\nph}(x_j)$,  $E_N^{{\nph}}(x_j)$ is computed from the density $\displaystyle\rho_N^{\nph} = (2\pi)^{\frac{1}{4}} m^{\nph}_0$. We also compute the characteristic variables 
		\[
		\bw^{n+1,*}(x) = R^{-1}\bm^{n+1,*}(x).
		\]
		\item Compute 
		\[
		w^{n+1}_k(x_j) = w^{n+1,*}_k(x_j-\tfrac{\Dt}{2}\lambda_k), \ k = 0,1,\cdots N ~~\Longrightarrow ~~\bm^{n+1}(x_j) = R\bw^{n+1}(x_j).
		\]
	\end{itemize}
	$\bm^{n+1}(x_j)$ is then taken as the approximation to $\bm(x_j,t^{n+1})$. In the first and third steps, we use linear interpolation to obtain the nodal values at $x_j-\frac{\Dt}{2}\lambda_k$. The backward integration of the adjoint PDE \eqref{eq:cont adjoint pde} can be derived in a similar manner.
	
	\subsection{Parameter update} \label{sec:momentum gradient descent}
	To update the parameters, we employ the gradient descent method with momentum to accelerate the convergence. Starting with zero momentum $w^{(0)} = 0$, we update $\balpha$ with
	\begin{equation}\label{eq:GD momentum}
		\begin{split}
			\left\{\begin{array}{l}
				w^{(i+1)} = \beta w^{(i)}+\nabla_\balpha L^{(i)}  \\
				\\
				\balpha^{(i+1)} = \balpha^{(i)} -\eta^{(i+1)} w^{(i+1)} 
			\end{array}\right., \quad i = 0,1,2, \cdots,
		\end{split}
	\end{equation}
	where the weight $\beta\in [0,1)$ measures the strength of inertia, $\eta^{(i)}$ is the learning rate at the $i-$th iteration. Such acceleration protocol dates back to Polyak's heavy ball method \cite{polyak1964some}. A large value of $\beta$ indicates that the iterations are strongly affected by the previous updates. When $\beta = 0$, it reduces to the vanilla gradient descent. 
	
	To determine the appropriate learning rate for each parameter, we further apply Jacob's scheme \cite{jacobs1988increased} for step size adaptation. Let $\alpha_k^{(i)}$ be a single weight of the cost function $L(\balpha^{(i)})$ at the $i-$th iteration, the corresponding learning rate, $\eta_k^{(i)}$, is updated with the rule
	\begin{equation}\label{eq:jacobs adaptation}
		\eta^{(i+1)}_k = \left\{\begin{array}{ll}
			\eta_k^{(i)}+\kappa &, \overline{\delta}_k^{(i-1)}\cdot \delta_k^{(i)}>0 \\
			(1-\gam)\eta_k^{(i)} &, \overline{\delta}_k^{(i-1)}\cdot \delta_k^{(i)}<0
		\end{array}
		\right., \quad \kappa>0, \quad 0<\phi<1,
	\end{equation}
	where 
	\[
	\delta_k^{(i)} = \frac{\partial L(\balpha^{(i)})}{\partial \alpha_k^{(i)}}, \ \ \text{and} \ \ \overline{\delta}^{(i)}_k = (1-\th)\delta_k^{(i)}+\th\overline{\delta}^{(i-1)}_k, \quad 0<\theta<1.
	\]
	Then the \emph{component-wise} learning rate $\eta^{(i+1)}$ in \eqref{eq:GD momentum} is the diagonal matrix with the $(k,k)-$th component equal to $\eta_k^{(i+1)}$.
	
	In our computation, we take $\beta = 0.9$, $\gam = 0.3$, $\th = 0.7$. The learning rate is uniformly initialized with $\eta^{(1)}_k = \eta_0$, where $\eta_0$ is tuned to ensure an efficient and stable convergence. The increase rate of step size $\kappa$ is set to $\eta_0/10$. The initial guess of parameters is set to $\balpha = \mathbf{0}$ for all test cases. Iterations are stopped when $||\nabla_{\balpha}L||_{\infty}<10^{-3}$ or the maximum iteration of 1000 steps is reached, whichever comes first.

	\section{Numerical experiments}\label{sec:numerical experiment}
	In this section, we present two numerical experiments to demonstrate that instabilities in plasma can be effectively controlled by the external electric field derived from the moment-based optimization problem \eqref{eq:optimization with truncated moments}. Additionally, we compare our results with those obtained by controlling instabilities through the full Vlasov-Poisson system-based optimization problem \eqref{eq:optimization with VP}.
	
	\subsection{Suppressing two-stream instability} 
	We consider first the two-stream distribution, which is known to be an unstable equilibrium:
	\[
	\mu(v) = \frac{1}{2\sqrt{2\pi}}\exp\left(-\frac{1}{2}(v-\overline{v})^2\right)+\frac{1}{2\sqrt{2\pi}}\exp\left(-\frac{1}{2}(v+\overline{v})^2\right).
	\]
	Here we take $\overline{v} = 2.4$. The perturbed initial distribution is set to
	\[
	f_0(x,v) = (1+\eps\cos(0.2 x))\mu(v), \quad (x,v)\in[0,10\pi]\times[-8,8]\,,
	\]
	where perturbation strength is $\eps = 10^{-3}$. The background ion density is set to $\rho_{ion} = 1$. Periodic boundary conditions are applied at both ends of the spatial interval. The computational domain is discretized by the uniform mesh with $N_x = 100$ and $N_v = 200$ grids in the space and velocity directions. 
	
	To suppress the plasma instability, we introduce a time-independent external field taking the form
	\begin{equation}\label{eq:H expansion}
		H(x) = \sum^K_{k = 1}\a_k \sin(\tfrac{k}{5}x)+ \sum^K_{k=0}\b_k \cos(\tfrac{k}{5}x)\,,
	\end{equation}
	where $K=10$ is chosen to be 10 in the following experiments. We then obtain the optimal $\balpha = \{\alpha_k,\beta_k\}$ by solving the approximate optimization problem \eqref{eq:optimization with truncated moments} . The approximate electric field $E_N(x,t)$ of the truncated moment system \eqref{eq:optim constraint} is solved by
	\[
	\left\{\begin{array}{l}
		E_N = -\partial_x \phi_N, \quad -\partial_x^2\phi_N = \rho_N-1, \quad \rho_N(x,t) = (2\pi)^{\frac{1}{4}}m_0(x,t)  \\
		\\
		\phi_N(0) = \phi_N(10\pi) = 0.
	\end{array}
	\right.
	\]
	In the test, we truncate \eqref{eq:optim constraint} at order $N = 30$, and the moment perturbation \eqref{eq:optim target} is minimized at terminal time $T = 30$. The time step of semi-Lagrangian schemes is decided by the following condition
	\[
	\frac{\max_{0\leq k\leq N}|\lambda_k| \Dt}{\Dx} = 3\,.
	\]
	We denote the resulting external field as $H_{mom}$.
	
	To verify the effectiveness of $H_{mom}$ in terms of controlling plasma instability as applied to the true physics model \eqref{eq:1d vlasov poisson}, we substitute $H_{mom}$ into the original Vlasov-Poisson equation and calculate the resulting particle distribution $f(x,v,t;H_{mom})$. Correspondingly,  the electric field $E(x,t)$ with respect to the VP system is calculated by
	\[
	\left\{\begin{array}{l}
		E = -\partial_x \phi, \quad -\partial_x^2\phi = \rho-1, \quad \rho(x,t) = \int f(x,v,t)\rd v,  \\
		\\
		\phi(0) = \phi(10\pi) = 0.
	\end{array}
	\right.
	\]
	In our experiments, the Vlasov-Poisson equation is discretized with the semi-Lagrangian scheme following the lines of \cite{einkemmer2024suppressing}, with the mesh size $N_x = 100$, $N_v = 200$, and the time step $\Dt = 0.1$.
	
	Figure \ref{fig:two stream perturb} shows the evolution of particle perturbation, $\displaystyle J(t) = \frac{1}{2}||f(t,\cdot,\cdot)-\feq||^2_2$, and electric energy, $\energy_e(t) = \displaystyle \frac{1}{2}\int E(x,t)^2 \rd x$, up to $t = 40$. It is seen that when the external field is absent, the initial perturbation is rapidly amplified as time develops, which leads to the deviation of $f$ from $\mu$, as well as the exponential increase in the electric energy. With the help of $H_{mom}$, such an instability is markedly reduced within the optimization time interval $t\in [0,30]$. At the terminal time of optimization $T = 30$, $H_{mom}$ suppresses the $L^2-$perturbation from $J(30)\approx 0.23$ to $2.71\times10^{-5}$, and the electric energy from $\energy_e(30) \approx 0.75$ to $1.31\times10^{-4}$. As time evolves further to $t = 40$, $H_{mom}$ can still impose effective control on the instability, reducing $J(40)$ from $0.45$ to $6.96\times10^{-4}$ and $\energy_e(40)$ from $1.78$ to $2.86\times 10^{-3}$. The suppression of perturbation can also be seen from the particle distributions shown in Figure \ref{fig:two stream f}.
	\begin{figure}[h!]
		\centering
		\begin{subfigure}{0.4\textwidth}
			\includegraphics[scale = 0.25]{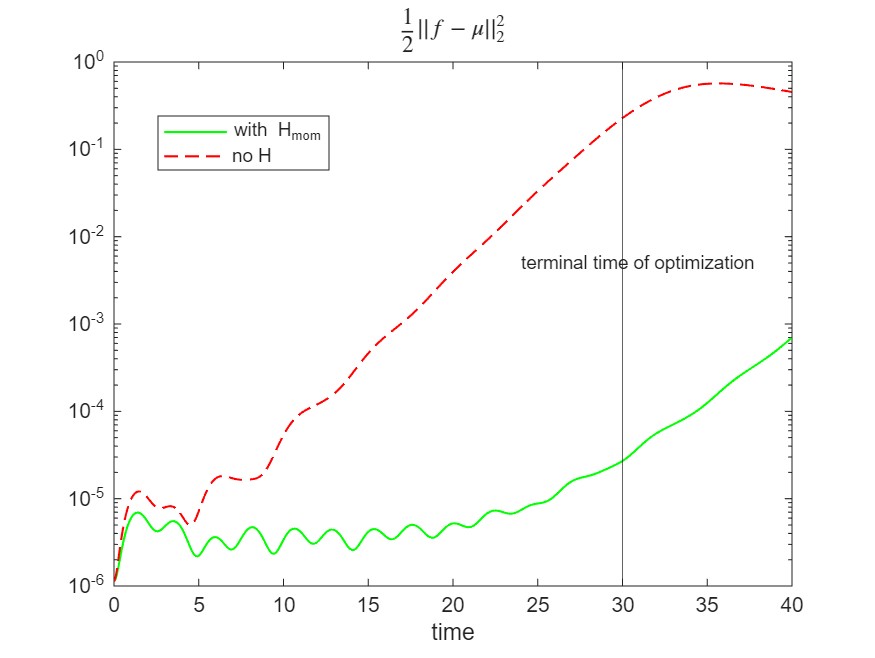}
			\subcaption{$L^2-$perturbation}
		\end{subfigure}
		\quad 
		\begin{subfigure}{0.4\textwidth}
			\includegraphics[scale = 0.25]{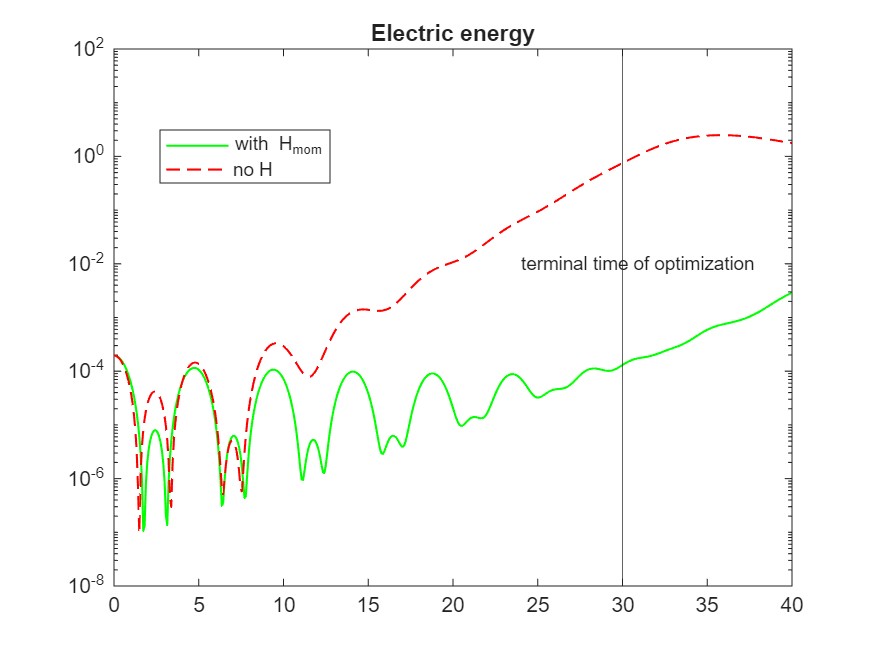}
			\subcaption{electric energy}
		\end{subfigure}
		\caption{Two stream instability. History of perturbation with respect to $H_{mom}$ optimized against $N = 30$ moments. Green lines represent the solutions generated by time-independent $H_{mom}$. Red lines represent the solutions without external field.}
		\label{fig:two stream perturb}
	\end{figure}
	
	\begin{figure}[h!]
		\centering
		\begin{subfigure}{0.3\textwidth}
			\includegraphics[scale = 0.24]{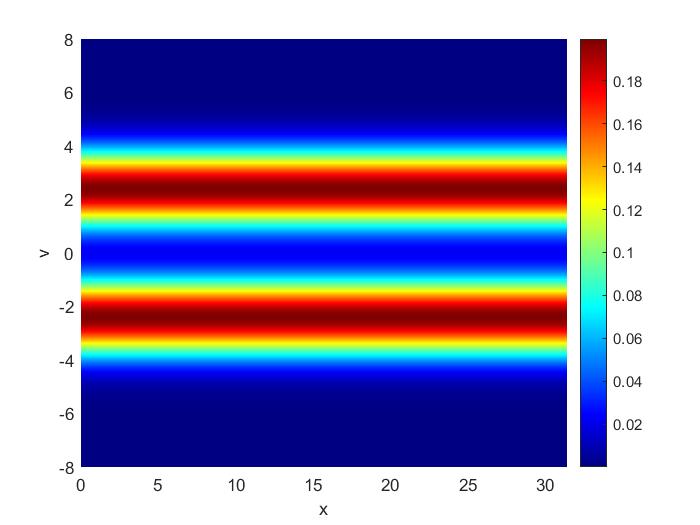} 
			\subcaption{target equilibrium $\feq(v)$}
		\end{subfigure}
		\quad
		\begin{subfigure}{0.3\textwidth}
			\includegraphics[scale = 0.24]{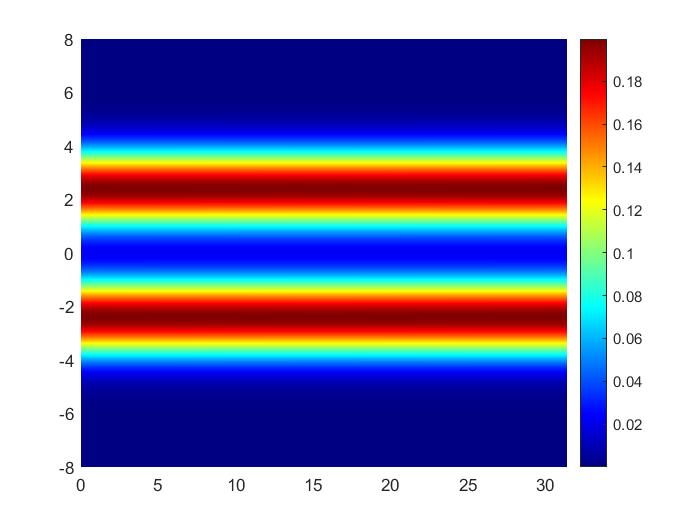}
			\subcaption{$f(x,v,0)$}
		\end{subfigure}
		\\
		\begin{subfigure}{0.3\textwidth}
			\includegraphics[scale = 0.24]{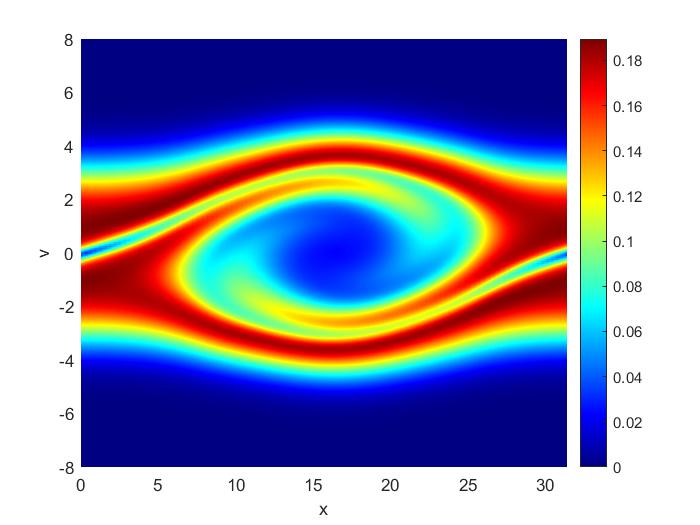} 
			\subcaption{$f(x,v,40)$ with $H = 0$}
		\end{subfigure}
		\quad
		\begin{subfigure}{0.3\textwidth}
			\includegraphics[scale = 0.24]{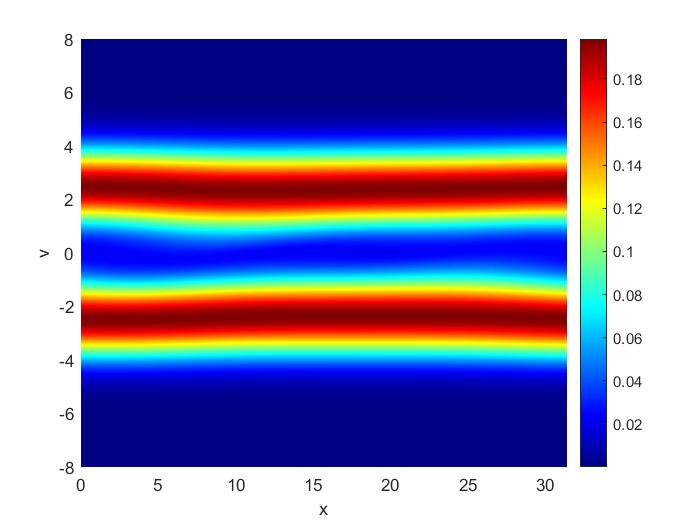} 
			\subcaption{$f(x,v,40)$ with $H_{mom}$}
		\end{subfigure}
		\caption{Two stream instability. Comparison of particle distributions. $H_{mom}$ is optimized through moment-based optimization \eqref{eq:optimization with truncated moments} with $N = 30$.}
		\label{fig:two stream f}
	\end{figure}
	
	Next, we compare the results generated by $H_{mom}$ optimized against different numbers of moments in the hydrodynamic optimization \eqref{eq:optimization with truncated moments}. As a reference, we also compute the external field $H_{VP}$ directly optimized against the original kinetic formulation \eqref{eq:optimization with VP}, that is, the moments are not truncated. The framework of semi-Lagrangian optimization for $H_{VP}$ follows the lines of \cite{einkemmer2024suppressing}, and the parameter update still follows the momentum gradient descent in Section \ref{sec:momentum gradient descent}. As shown in Figure \ref{fig:two stream compare N}, increasing the number of moments included in the optimization leads to a more effective control of the perturbation by the resulting optimized field $H_{mom}$ within the optimization time interval. This is somewhat expected, as the moment system is truncated at a higher order, the objective function \eqref{eq:optim target} provides a tighter upper bound on  $\frac{1}{2}||f(T)-\feq||^2_2$, and the constraint equations \eqref{eq:optim constraint} more accurately approximate the true Vlasov–Poisson dynamics near equilibrium. As a result, the performance of $H_{mom}$ becomes increasingly similar to that of the optimal control field $H_{VP}$. This can also be seen from the optimized external fields presented in Figure \ref{fig:two stream compare N}. However, it is important to note that solving a higher-order moment system incurs increased computational cost, which should be taken into account in practical applications.
	%\lw{How many iterations do you need for moment-control and VP-control? }\jl{For two-stream instability test case, VP-control takes 139 iters under lr = 1e-4, moment-control with $N = 30$ takes 594 iters under lr = 1e-7. But their gradients, optimization settings, and stop criteria are quite different, I don't see how we can make a fair comparison. For $N = 15$ and $N = 20$ the iterations are somehow less stable and reach the maximum of 1000 iterations. This is probably due to gradient approximation, and using PyTorch autodifferentiation does not seem to help. For bump-on-tai, if I remember correctly, all the methods reach the maximum of 1000 iterations as the gradient get stuck at a certain level. }
	
	\begin{figure}[h!]
		\centering
		\begin{subfigure}{0.4\textwidth}
			\includegraphics[scale = 0.25]{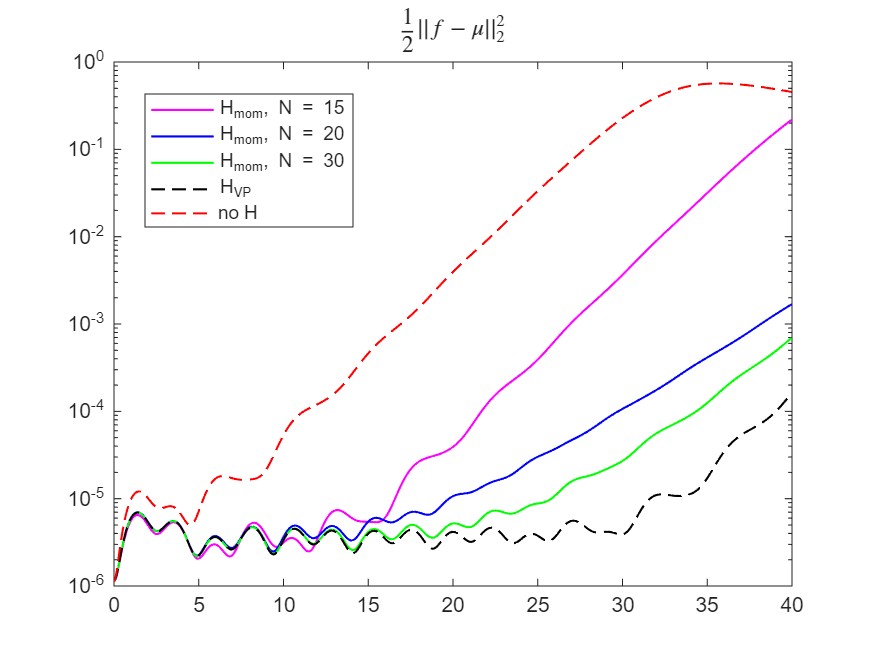}
			\subcaption{$L^2-$perturbation}
		\end{subfigure}
		\quad 
		\begin{subfigure}{0.4\textwidth}
			\includegraphics[scale = 0.25]{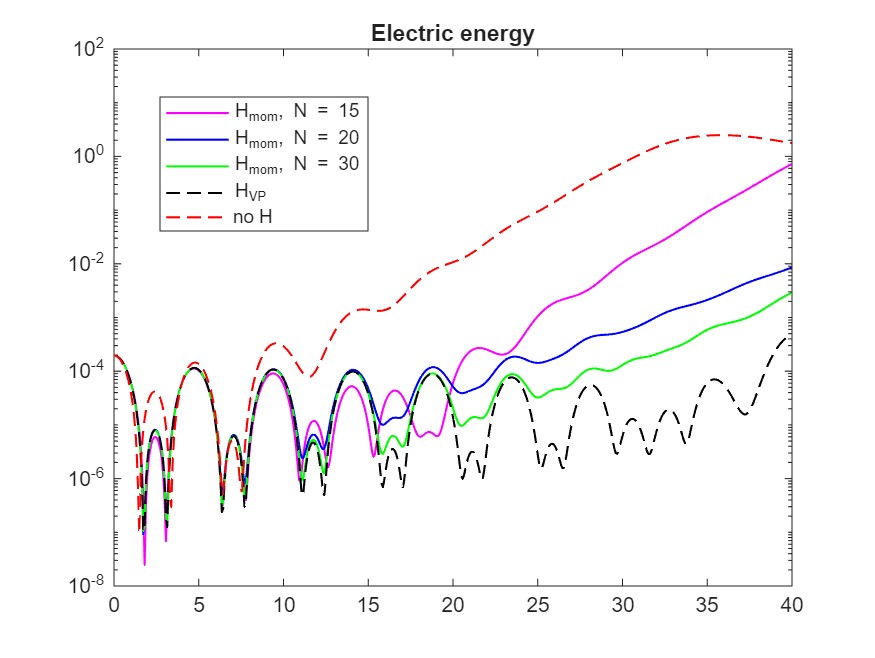}
			\subcaption{electric energy}
		\end{subfigure}
		\caption{Two stream instability. History of perturbation with respect to $H_{mom}$ optimized against different numbers of moments. As the number of moments increases, the result of moment-based optimization approaches that generated by $H_{VP}$ (black dash lines), which indicates the use of full moments.}
		\label{fig:two stream compare N}
	\end{figure}
	
	\begin{figure}[h!]
		\centering
		\includegraphics[scale = 0.25]{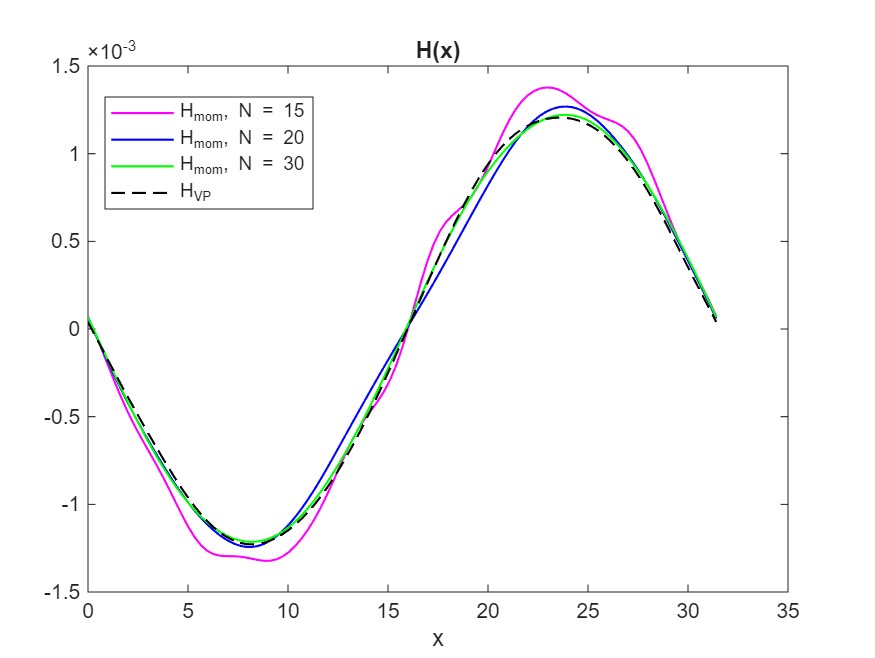}
		\caption{Two stream instability. Optimized external fields with respect to different numbers of moments.}
		\label{fig:two stream compare H}
	\end{figure}
	
	%\jl{
		%In my experiment, the optimized $H_{VP}$ reduces $J(40) = \frac{1}{2}||f(\cdot,\cdot,40)-\mu||^2_2$ to $1.62\times 10^{-4}$, which is better than the result in the earlier work \cite{einkemmer2024suppressing} ($J(40)\approx 2.4\times10^{-3}$). Besides the difference in optimization settings, I believe the choice of basis should be the more essential reason -- in contrast to \cite{einkemmer2024suppressing} where $H$ was only constructed cosine modes, my $H_{VP}$ here is dominated by $-0.0012\sin(0.2x)$. Indeed, in our test case the initial density perturbation is a cosine function $\delta \rho = 10^{-3}\cos(0.2x)$, then the electric field $E = \int^x \delta \rho$ would produce sine mode perturbations, which need to be canceled by sine modes in the external field.
		%}
	
	\subsection{Suppressing Bump-on-tail instability} 
	We next examine the validity of moment based control  \eqref{eq:optimization with truncated moments} for the bump-on-tail instability. More specially, the equilibrium we consider here takes the form: 
	\[
	\feq(v) = \frac{\omega_1}{\sqrt{2\pi}}\exp\left(-\frac{1}{2}v^2\right)+\frac{\omega_2}{\sqrt{2\pi v_t}}\exp\left(-\frac{1}{2v_t}(v-u)^2\right)
	\]
	with $\omega_1 = 0.8$, $\omega_2 = 0.2$, $u = 3.5$, $v_t = 0.5$. The perturbed initial state is set to
	\[
	f_0(x,v) = (1+\eps\sin(0.2 x))\feq(v), \quad (x,v)\in[0,10\pi]\times[-8,8].
	\]
	with $\eps = 10^{-3}$. We suppress the instability by introducing external field $H_{mom}$ of the form \eqref{eq:H expansion} optimized through the problem \eqref{eq:optimization with truncated moments} with  $T = 25$ and different numbers of moments \footnote{For $N = 15$ and $N = 20$, the iterations do not converge stably with the gradient approximation \eqref{eq:approx grad}, thus their computations are implemented with PyTorch autodifferentiation}. $K = 10$ modes are employed to construct $H$. Again, the computational results displayed in Figures \ref{fig:bump on tail perturb}--\ref{fig:bump on tail f} verify that introducing $H_{mom}$ can help suppress the perturbation growth effectively. The improvement in control achieved by using more moments in the optimization is also confirmed. The optimized external fields are presented in Figure \ref{fig:bump on tail H}.
	
	\begin{figure}[h!]
		\centering
		\begin{subfigure}{0.4\textwidth}
			\includegraphics[scale = 0.25]{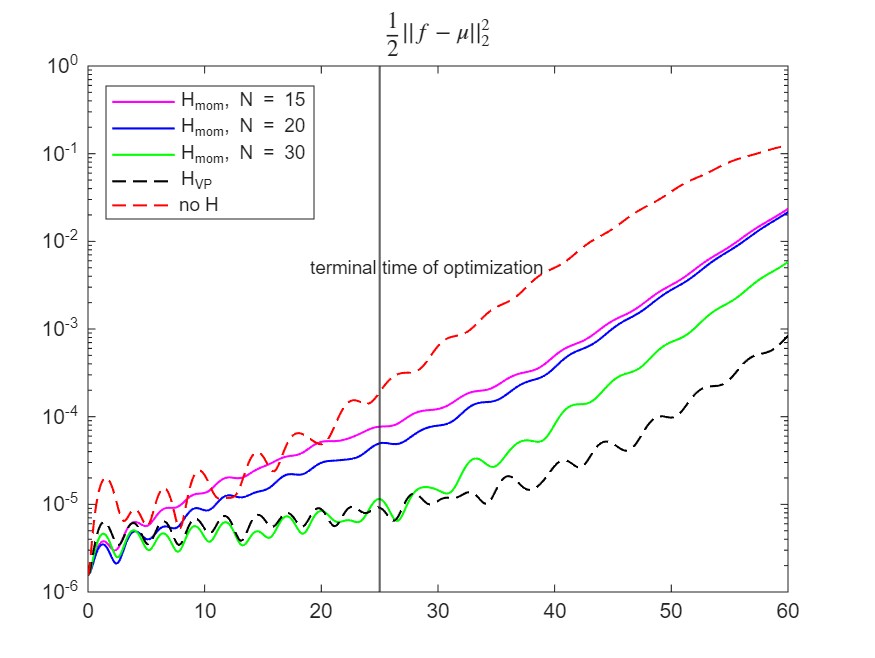}
			\subcaption{$L^2-$perturbation}
		\end{subfigure}
		\quad 
		\begin{subfigure}{0.4\textwidth}
			\includegraphics[scale = 0.25]{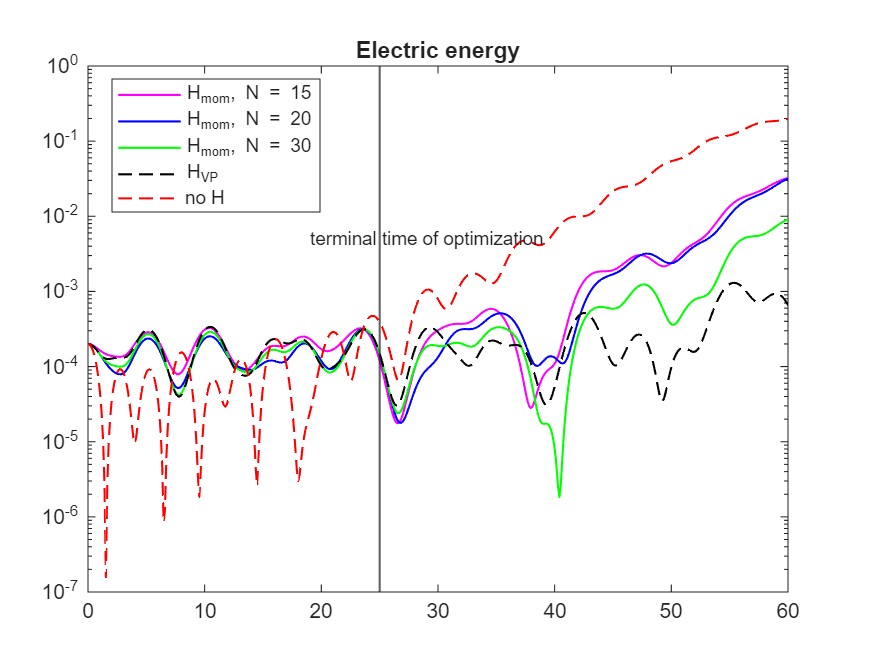}
			\subcaption{electric energy}
		\end{subfigure}
		\caption{Bump-on-tail instability. History of perturbation with respect to $H_{mom}$ optimized against different numbers of moments. Green lines represent the solutions generated by time-independent $H_{mom}$. Red lines represent the solutions without external field.}
		\label{fig:bump on tail perturb}
	\end{figure}
	
	\begin{figure}[h!]
		\centering
		\begin{subfigure}{0.3\textwidth}
			\includegraphics[scale = 0.24]{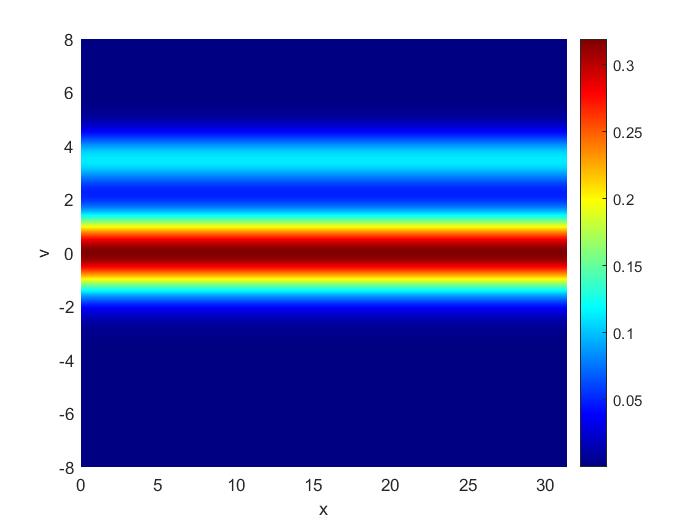} 
			\subcaption{target equilibrium $\feq(v)$}
		\end{subfigure}
		\quad
		\begin{subfigure}{0.3\textwidth}
			\includegraphics[scale = 0.24]{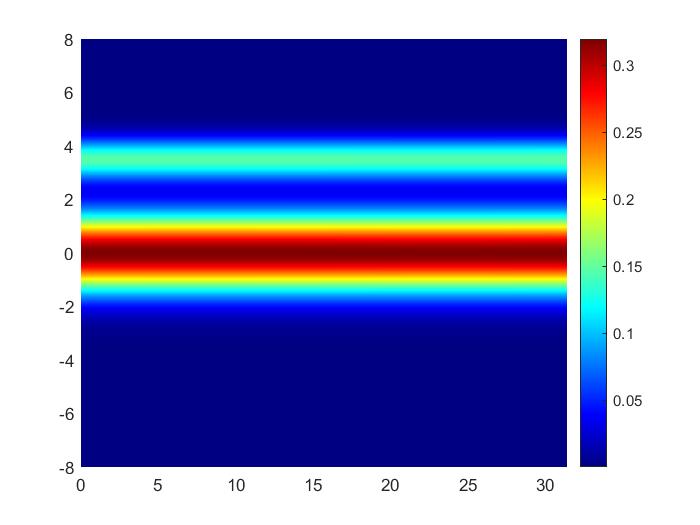}
			\subcaption{$f(x,v,0)$}
		\end{subfigure}
		\\
		\begin{subfigure}{0.3\textwidth}
			\includegraphics[scale = 0.24]{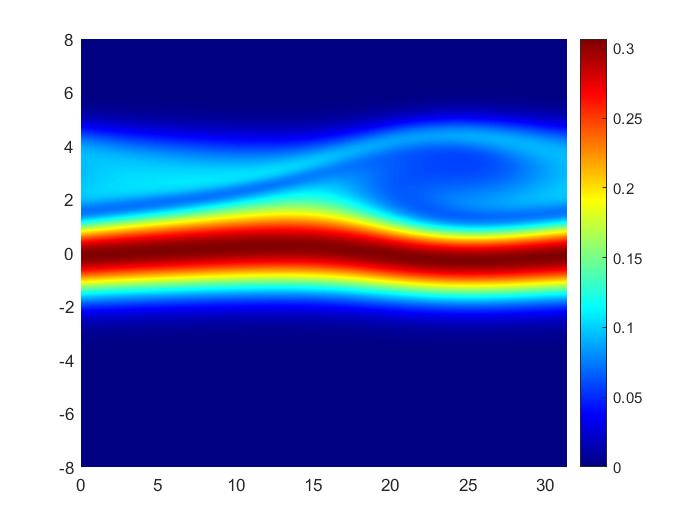} 
			\subcaption{$f(x,v,60)$ with $H = 0$}
		\end{subfigure}
		\quad
		\begin{subfigure}{0.3\textwidth}
			\includegraphics[scale = 0.24]{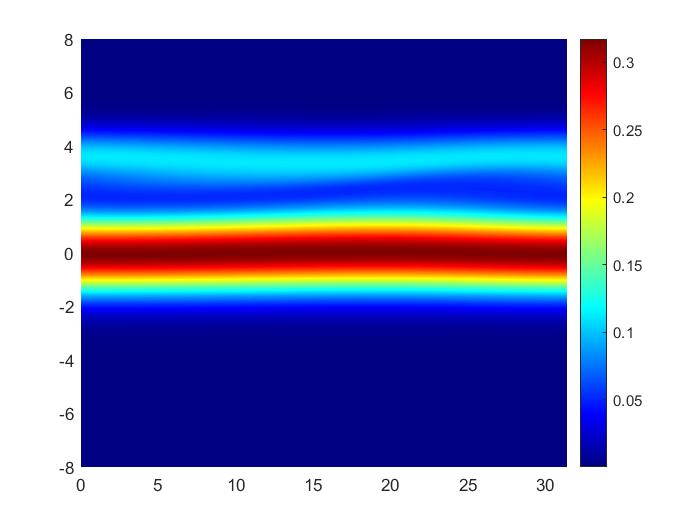} 
			\subcaption{$f(x,v,60)$ with $H_{mom}$}
		\end{subfigure}
		\caption{Bump-on-tail instability. Comparison of particle distributions. $H_{mom}$ is optimized through moment-based optimization \eqref{eq:optimization with truncated moments} with $N = 30$.}
		\label{fig:bump on tail f}
	\end{figure}
	
	\begin{figure}
		\centering
		\includegraphics[scale = 0.25]{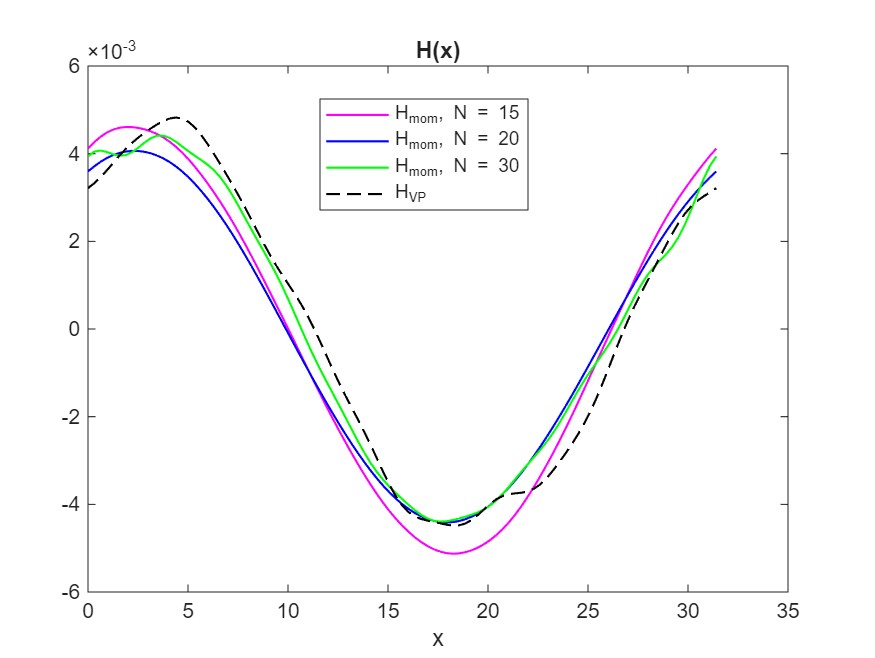}
		\caption{Bump-on-tail instability. Optimized external field with respect to different numbers of moments.}
		\label{fig:bump on tail H}
	\end{figure}

	\section*{Acknowledgment}
	JL and JC were supported by NSF-CCF:2212318, and JC was additionally supported by an Albert
	and Dorothy Marden Professorship. 
	LW is supported in part by NSF grant DMS-1846854, DMS-2513336, and the Simons Fellowship. JL and LW would like to thank the fruitful discussion with Professors Michael Herty and Lukas Einkemmer.

	\bibliographystyle{siam}
	\bibliography{Ref.bib, Ref_controlVP.bib}
	
\end{document}